\documentclass[11pt,reqno]{amsart}

\usepackage{amsfonts,amssymb,amsmath,amsopn,amsthm,graphicx}
\usepackage{amsxtra, mathrsfs}
\usepackage{amsbsy,dsfont}
\usepackage{colordvi}
\usepackage[usenames,dvipsnames]{color}
\usepackage{mathtools}
\usepackage[colorlinks=true]{hyperref}
\usepackage{enumitem}
\usepackage{bbm} 
\usepackage[parfill]{parskip}    
\usepackage{lmodern} 

\usepackage{blindtext}
\usepackage{hyperref}

\newcommand{\N}{\mathbb{N}}

\newcommand{\R}{\mathbb{R}}

\newcommand{\id}{\mathds{1}}
\newcommand{\Sph}{\mathbb{S}}

\newtheorem{theorem}{Theorem}[section]
\newtheorem{corollary}[theorem]{Corollary}
\newtheorem{lemma}[theorem]{Lemma}
\newtheorem{proposition}[theorem]{Proposition}

\theoremstyle{definition}

\newtheorem{remark}[theorem]{Remark}

\numberwithin{equation}{section}

\begin{document}

\title[Optimizing the ground of a Robin Laplacian: asymptotic behavior]{Optimizing the ground of a Robin Laplacian:\\ asymptotic behavior }

\author{Pavel Exner}
\address[Pavel Exner]{Doppler Institute for Mathematical Physics and Applied Mathematics, Prague  Czechia, \and \\ Department of Theoretical Physics, Nuclear Physics Institute, Czech Academy of Sciences, 
\v Re\v z near Prague, Czechia}
\email{exner@ujf.cas.cz}

\author {Hynek Kova\v{r}\'{\i}k}
\address [Hynek Kova\v{r}\'{\i}k]{DICATAM, Sezione di Matematica, Universit\`a degli studi di Brescia, Italy}
\email {hynek.kovarik@unibs.it}

\thanks{}

\begin{abstract}
In this note we consider achieving the largest principle eigenvalue of a Robin Laplacian on a bounded domain $\Omega$  by optimizing the Robin parameter function under an integral  constraint.
The main novelty of our approach lies in establishing  a close relation between the problem under consideration and the asymptotic behavior of the Dirichlet heat content of $\Omega$.  By using this relation we deduce a two-term
asymptotic expansion of the principle eigenvalue and discuss several applications. 
\end{abstract}


\maketitle




\section{\bf Introduction}

\noindent
Let $\Omega\subset\R^N$ be an open bounded connected  smooth domain.  Given a function $\sigma \in L^1(\partial\Omega)$ we consider the functional $Q[u, \sigma]$ on $H^1(\Omega)$ defined by
\begin{equation} \label{form}
Q[u, \sigma]= \frac{\int_\Omega |\nabla u|^2 \,dx + \int_{\partial\Omega} \sigma\, |u|^2\, d\nu}{\| u\|^2_{L^2(\Omega)} } 
\end{equation}
if the right hand side is finite and by $Q[u, \sigma] = \pm \infty$ otherwise. Here
 $d\nu$ denotes the surface measure on $\partial\Omega$. Let 
\begin{equation} \label{lambda}
\lambda(\sigma) = \inf_{ \substack{u \in H^1(\Omega) \\ u \neq 0}} Q[u, \sigma]. 
\end{equation} 
By standard variational arguments it follows that  $\lambda(\sigma)$ is the principal eigenvalue of the Laplace operator in $L^2(\Omega)$ subject to  Robin boundary conditions formally given by
$$
\partial_n u + \sigma u =0 \quad \text{on \ } \partial\Omega,
$$
where $\partial_n$ denotes the outer normal derivative. 
For a given $\mu\in\R$ we define
\begin{equation} \label{sigma-mu}
\Sigma_\mu := \big\{ \sigma \in L^1(\partial\Omega): \ \int_{\partial\Omega} \sigma = \mu \big\}.
\end{equation}

The main object of our interest is the quantity
\begin{equation} \label{supinf}
\Lambda_\mu = \sup_{\sigma\in \Sigma_\mu} \lambda(\sigma),
\end{equation} 
and its asymptotic behavior for $\mu\to -\infty$.

The problem of analyzing the asymptotic behavior of $\lambda(\sigma)$ for a  $\sigma$ {\em constant} on $\partial\Omega$ and negatively diverging has attracted a huge interest in the literature. 
Particular attention has been paid to the link between the geometry of $\partial\Omega$ and the coefficients of the associated  asymptotic expansion.
In \cite{luzhu} it was proved that for  $C^1$ domains one has  $\lambda(\sigma)=-\sigma^2+o(\sigma^2)$. 
A two-term  expansion
\begin{equation} \label{sigma-const-lin}
\lambda(\sigma)=-\sigma^2 +\sigma (N-1)H_{\max}+o(\sigma) \qquad  \sigma\to -\infty,
\end{equation}
with $H_{\max}$  denoting the maximal mean curvature of $\partial\Omega$, 
was obtained, under additional regularity assumptions on $\Omega$, first for $N=2$ in \cite{p13,emp}, and later in \cite{pp15,pp15b} in any dimension.
Under further geometric hypotheses on $\Omega$, higher terms in the expansion \eqref{sigma-const-lin} were found in \cite{dk,HK,pp15b}. Analogous problem on non-smooth domains, in particular on domains with corners and cusps, were studied as well, cf.~\cite{lp,bp, kh,kop,kop2,v24, p24}. Various estimates on $\lambda(\sigma)$ were obtained in \cite{bfnt, kov, fk}. Let us also mention that a modified version of \eqref{sigma-const-lin} for the $p-$Laplace operator with Robin boundary conditions was established in \cite{kop1}, see also \cite{p20}.

The question that we address here is different: What happens when $\lambda(\sigma)$ with a constant $\sigma$ is replaced by \eqref{supinf}, in other words when we maximize $\lambda(\sigma)$ under 
the natural constraint \eqref{sigma-mu}? Our main result, Theorem \ref{thm-sup}, provides a complete answer to this question and shows that \eqref{supinf} has a unique maximizer for any $\mu\in\R$. Moreover, we derive an explicit 
expression for $\Lambda_\mu$ and for the (unique) optimizing $\sigma_\mu$ in terms of the Green's function of the Dirichlet Laplacian on $\Omega$, see equations \eqref{us} and \eqref{max}. 
This is achieved by 
a suitable modification of the construction  previously used only for $\mu >0$ in \cite{kov}, see also \cite{cu, dgk}.

Furthermore, we deduce from Theorem \ref{thm-sup} that, as $\mu\to-\infty$, 
\begin{equation} \label{Lambda-limit}
\Lambda_\mu = -\frac{\mu^2}{|\partial\Omega|^2}  + \mu\  \frac{N-1}{|\partial\Omega|^2}\, \int_{\partial\Omega} H\,  d\nu +  \mathcal{O}(1)
\end{equation}
where $H$ denotes the mean curvature of the boundary, see Corollary \ref{thm-smooth}.

Let us put our result into perspective and compare it with the results obtained previously for constant $\sigma$ cited above. Note that the first term on the right hand side of \eqref{Lambda-limit} is given by the square of the mean value of $\sigma$. This is compatible with the first term on the right hand side of \eqref{sigma-const-lin}.  
On the other hand, the coefficient of the second term in \eqref{Lambda-limit} is proportional to the mean value of $H$, whereas in \eqref{sigma-const-lin} we have the maximum of $H$. The reason for this difference lies in the different behavior of the associated minimizers. While the minimizer of \eqref{lambda} with  constant $\sigma$ concentrates, as $\sigma \to -\infty$,
near the set at which the mean curvature attains its maximum, see \cite[Thm.~6.5]{kop1}, the minimizer of \eqref{lambda} with $\sigma=\sigma_\mu$ is {\em constant on $\partial\Omega$}. The latter fact is the 
starting point of the proof of our main result. Indeed, it turns out that a sufficient condition for $\sigma$ to be an optimizer in \eqref{supinf} is that the corresponding minimizer of \eqref{lambda} is constant at the boundary, see equation \eqref{test-f}.

We therefore construct a positive solution of the eigenvalue equation in such a way that its restriction to the boundary of $\Omega$ is equal to $1$,  cf.~\eqref{minimizer}. In Theorem \ref{thm-sup} we then show that this solution is the unique minimizer of \eqref{lambda} associated to the optimizing function $\sigma_\mu$, and find an explicit equation for $\Lambda_\mu$. By exploiting the connection with the semi-classical behavior of the heat content of $\Omega$ we then derive the asymptotical expansion \eqref{Lambda-limit}, see Corollary \ref{thm-smooth}.  Several examples are discussed in Remark \ref{rems}. In the closing Section \ref{sec-related} we present some further
related results.
 
 \begin{remark}
If $\partial\Omega$ is $C^1$ smooth, then  a trace operator is well defined on $H^1(\Omega)$. More precisely, we have
\begin{equation} \label{trace-imbed}
\| u\|_{L^q(\partial\Omega)} \, \leq\, C \, \| u\|_{H^1(\Omega)}, \qquad \forall\ q \, \left\{
\begin{array}{l@{\qquad}l}
\leq \, 2(N-1)/(N-2)    &       {\rm if \ } \   N>2,  \\
< +\infty    &    {\rm if \ } \  N=2
\end{array}
\right. 
\end{equation}
with a compact imbedding. 
If  $\sigma \in L^p(\partial\Omega)$ with  $p \geq \frac{2(N-1)}{N}$ for $N\geq 3$ and $p>1$ for $N=2$, then the imbedding \eqref{trace-imbed} ensures that $Q[u, \sigma]$ is bounded from for any $\sigma\in\Sigma_\mu$ and any $\mu\in\R$. 
\end{remark}


\section{\bf Main results}

In this section we assume throughout that $\Omega$ is a bounded and connected set in $\R^N$ with a $C^3$ boundary.

Before we state our main results let us introduce some necessary notation. We denote by $-\Delta^D_\Omega$ the Laplace operator on $\Omega$ with homogenous Dirichlet boundary conditions. The volume and the surface measure of $\Omega$ are denoted by  $|\Omega|$ and $|\partial\Omega|$ respectively. The surface measure of the $N-$dimensional unit sphere will be denoted by  $|\Sph_N| $. Finally,  we denote by $B(x,r)$ the open ball of radius $r$ centered in $x\in \R^N$.

\subsection{The supremum}Let $E_1 < E_2\leq \dots E_j \leq \dots$ and $\varphi_j$ be the eigenvalues and normalized eigenfunctions of  $-\Delta^D_\Omega$ and recall that $E_1>0$. 
We need a series of auxiliary results. Let us start by 
recalling the existence of a minimizer in \eqref{form}. 

\begin{lemma} \label{lem-minimiser}
Let $\sigma \in \Sigma_\mu$ and suppose that $\sigma \in L^p(\partial\Omega)$ with  $p \geq \frac{2(N-1)}{N}$ for $N\geq 3$ and $p>1$ for $N=2$.
Then the functional $Q[\sigma, \cdot\, ]$ admits a positive minimiser $\psi\in H^1(\Omega)$ which satisfies
\begin{equation} \label{euler-lagr}
-\Delta  \psi = \lambda(\sigma)\, \psi \quad \text{in \ } \Omega, \qquad \partial_n \psi + \sigma\, \psi =0 \quad \text{on \ } \partial\Omega.
\end{equation}
\end{lemma}

For the proof of Lemma \ref{lem-minimiser} we refer to \cite[Prop.~6.1]{kop1}. 

As already mentioned in the introduction, we will construct the minimizer in\eqref{euler-lagr} relative to the optimizing $\sigma$ explicitly, see \eqref{minimizer}. To do so, we  define,
for any $s \in (0,E_1)$, 
\begin{equation} \label{us}
U_s = (-\Delta^D_\Omega-s)^{-1}\,\id,
\end{equation}
where $\id$ denotes the function identically equal to $1$ on $\Omega$. In view of the regularity of $\Omega$ we have $U_s\in H^1_0(\Omega)\cap H^2(\Omega)$. Later we will show that $U_s\in L^\infty(\Omega)$, see Lemma \ref{lem-U_s} below. Moreover, $U_s >0$ in $\Omega$ since $(-\Delta^D_\Omega-s)^{-1}$ is positivity preserving. 
Together with $U_s$ we introduce the function $F:(-\infty, E_1) \to \R$ given by
\begin{equation} \label{eq-f}
F(s) = s^2 \int_\Omega U_s(x) \, dx + s\, |\Omega|.
\end{equation}  

Next we state a couple of properties of the function $F$ will be fundamental for the proof of the main result. 

\begin{lemma} \label{lem-f}
 $F$ is differentiable and strictly increasing. 
\end{lemma}

\begin{proof}
Following \cite{kov} we let 
$$
g(s)= \int_\Omega U_s(x)\, dx = (\id,\,  (-\Delta^D_\Omega-s)^{-1} \id )_{L^2(\Omega)},
$$ 
so that $F(s) = s^2 g(s)+s\, |\Omega|$. Since the resolvent $(-\Delta^D_\Omega-s)^{-1}$ is analytic in $s$ on $(-\infty,E_1)$, it follows that $g \in C^\infty(-\infty,E_1)$. Hence $F\in C^\infty(-\infty,E_1)$. 
Moreover, by \cite[Lem.~3.2]{kov}, we have 
$$
g'(s) = \|U_s\|^2_{L^2(\Omega)}.
$$
This implies
\begin{equation} 
F'(s) = 2s g(s) +s^2 \|U_s\|^2_{L^2(\Omega)} +|\Omega| = \int_{\Omega} (1+sU_s)^2\, dx >0,
\end{equation}
as claimed.
\end{proof}

\begin{lemma} \label{lem-f-2}
We have
\begin{equation} \label{lim-F}
\lim_{s\to -\infty} F(s) = -\infty.
\end{equation}
\end{lemma}

\begin{proof}
Let $e^{t \Delta_\Omega^D}(x,y), t>0,$ denote the integral kernel of the heat semi-group generated 
by $-\Delta_\Omega^D$. Observe that, for any $s\leq 0$,
\begin{equation} \label{g-Q}
 \int_\Omega U_s(x)\, dx = \int_0^\infty e^{s t}\,  Q_\Omega(t)\, dt,  
\end{equation}
where
\begin{equation} 
Q_\Omega(t) =\iint_{\Omega\times\Omega} e^{t \Delta_\Omega^D}(x,y)\, dx dy
\end{equation}
is the heat content of $\Omega$ corresponding to initial homogeneous temperature $u =1$. By \cite[Thm.~6.2]{vdbd}, 
\begin{equation} \label{heat-content-asymp}
\Big | Q_\Omega(t) - |\Omega| + \frac{2 |\partial\Omega|}{\sqrt{\pi}}\, t^{\frac12} \Big | \leq C\, t  \qquad \forall\, t >0,
\end{equation}
with a constant $C$ independent of $t$. 
Inserting this estimate into \eqref{g-Q} and using the identity 
$$
\int_0^\infty \sqrt{z}\  e^{-z}\, dz = \Gamma(3/2) = \frac{\sqrt{\pi}}{2},
$$
we get 
\begin{equation} \label{g-asymp}
 \int_\Omega U_s(x)\, dx  = -\frac{|\Omega|}{s}+ \frac{ |\partial\Omega|}{s \sqrt{|s|}}\,  +\mathcal{O}(s^{-2}) \qquad s\to -\infty.
\end{equation}
Now the claim follows from \eqref{eq-f}. 
\end{proof}

Lemma \ref{lem-f} implies that for every $\mu \in \R$ there exists a unique $s(\mu)$ such that
\begin{equation} \label{implicit}
F(s(\mu))  = \mu. 
\end{equation}
As we will prove shortly, it turns out that $s(\mu)$ given implicitly by \eqref{implicit} coincides with $\Lambda_\mu$, and that $u_\mu$ defined in \eqref{minimizer}  below is its corresponding eigenfunction.
It remains to show, however, that the latter is positive.  To this end we have to ensure that $u_\mu$  is positive in $\Omega$. This is done in the following

\begin{lemma} \label{lem-U_s}
For every $s\leq 0$ we have 
\begin{equation} \label{upperb-U_s}
\sup_{x\in\Omega} |s|\, U_s(x) < 1.
\end{equation}
\end{lemma}

\begin{proof} Let $s= -\kappa^2$. We
extend $(-\Delta^D_\Omega+\kappa^2)^{-1}(x,y)$ by zero to $\R^N\setminus\Omega$, and we deduce from Dirichlet monotonicity, see \cite{vdbd, gz}, 
that
\begin{equation} 
 (-\Delta^D_\Omega+\kappa^2)^{-1}(x,y) <  (-\Delta_{\R^N} +\kappa^2)^{-1}(x,y) \qquad  x,y \in \R^N.
\end{equation}
Here $\Delta_{\R^N}$ denotes the Laplace operator on $\R^N$. The Green function of the latter is given by 
\begin{equation} 
 (-\Delta_{\R^N} +\kappa^2)^{-1}(x,y) = \frac{1}{2\pi} \left(\frac{\kappa}{2\pi |x-y|} \right)^\nu\, K_\nu(\kappa|x-y|), 
\end{equation}
where $\nu = \frac N2-1$, and where $K_\nu$ denotes the modified Bessel function of the second kind, cf.~\cite[Chap.~9.6]{as}. 
From \cite[Eq.10.43.19]{dlmf} we learn that 
$$
\int_0^\infty t^{\nu +1} \, K_\nu(t) \, dt=  2^{\nu}\,  \Gamma(\nu +1) = 2^{\frac N2-1}\, \Gamma(N/2).
$$
Hence for any $\kappa\geq 0$ and $x\in\Omega$ we have
\begin{align*}
\kappa^2\, U_{\kappa^2}(x) & = \kappa^2\int_{\Omega}  (-\Delta^D_\Omega+\kappa^2)^{-1}(x,y) \, dy < \kappa^2\int_{\R^N}  (-\Delta_{\R^N}+\kappa^2)^{-1}(x,y) \, dy \\[5pt]
& =  \kappa^2\,  \frac{|\Sph_N|}{2\pi} \left(\frac{\kappa}{2\pi} \right)^\nu\,  \int_0^\infty r^{N/2} \, K_\nu(\kappa r)\, dr  =  \frac{|\Sph_N|}{(2\pi)^{N/2}} \int_0^\infty t^{\nu +1} \, K_\nu(t) \, dt=1,
\end{align*}
where, in the last step, we have used the identity
$$
|\Sph_N| = \frac{2\, \pi^{N/2}}{\Gamma(N/2)}\, .
$$
\end{proof}

\begin{theorem} \label{thm-sup}
Let $\Omega$ be as above. Then \eqref{supinf} is attained for any $ t\in \R$ and satisfies  
\begin{equation} \label{max}
\Lambda_\mu =   \lambda(\sigma_\mu) = s(\mu),  \quad \text{\rm where \ \ } \sigma_\mu = -s(\mu)\, \partial_n U_{s(\mu)}\big |_{\partial\Omega},
\end{equation}
and $s(\mu)$ is given by \eqref{implicit}.
Moreover, the optimizer $\sigma_\mu$ is unique in $\Sigma_\mu$.   
\end{theorem}

\begin{proof}
The case $\mu>0$ is treated in \cite[Thm.~3.3]{kov}. We therefore assume that $\mu\leq 0$. 
Let $\sigma_\mu$ be given as in \eqref{max}.  From \eqref{max} it follows that $\sigma_\mu\in L^\infty(\partial\Omega)$. Moreover, since $-\Delta U_{s(\mu)} = s(\mu) U_{s(\mu)}+1$, the divergence
theorem in combination with equations \eqref{max} and \eqref{implicit} gives
\begin{align*}
\int_{\partial\Omega} \sigma_\mu\, d\nu & = -\int_{\partial\Omega} \partial_n U_{s(\mu)}\, d\nu = -s(\mu) \int_\Omega \Delta U_{s(\mu)}\, dx = F(s(\mu)) =\mu .
\end{align*}
This shows that $\sigma_\mu\in\Sigma_\mu$.  Now let us set
\begin{equation} \label{minimizer}
u_\mu := s(\mu)\, U_{s(\mu)} + 1,
\end{equation}
so that
\begin{equation} \label{eq-ut}
-\Delta u_\mu = s(\mu)\, u_\mu \quad \text{in \ } \Omega, \qquad \partial_n u_\mu + \sigma_\mu =0 \quad \text{on \ } \partial\Omega.
\end{equation}
Moreover, by Lemma  \ref{lem-U_s} we have $u_\mu>0$ in $\Omega$.
We claim that $u_\mu$ is a minimiser of $Q[\sigma_\mu, \cdot\, ]$. Indeed, by \eqref{eq-ut} we have $Q[\sigma_\mu, u_\mu] = s(\mu)$. Assume that $\lambda(\sigma_\mu) < s(\mu)$. In view of Lemma \ref{lem-minimiser} and the regularity of $\sigma_\mu$  there exists a positive minimiser $\psi$ of $Q[\sigma_\mu, \cdot\, ]$ which satisfies equation \eqref{euler-lagr} with $\sigma=\sigma_\mu$. Hence $(\psi, u_\mu)_{L^2(\Omega)}=0$, which is in contradiction with the positivity of $\psi$ and $u_\mu$. We thus conclude that $\lambda(\sigma_\mu) = s(\mu)$.

To show that  $\sup_{\sigma\in \Sigma_\mu} \lambda(\sigma)=   \lambda(\sigma_\mu)$ we pick an arbitrary $\sigma\in\Sigma_\mu$. Using \eqref{lambda} and the fact that $u_\mu = 1$ on $\partial\Omega$ 
we obtain
\begin{equation} \label{test-f}
\lambda(\sigma)   \leq  Q[\sigma, u_\mu] = 
Q[\sigma_\mu, u_\mu] = \lambda(\sigma_\mu).
\end{equation}
To prove the uniqueness of $\sigma_\mu$ suppose that $\lambda(\sigma)=\lambda(\sigma_\mu)$ for some $\sigma\in\Sigma_\mu$. Mimicking  the above argument  we find 
$$
\lambda(\sigma_\mu)  =\lambda(\sigma) \leq  Q[\sigma, u_\mu] = Q[\sigma_\mu, u_\mu] = \lambda(\sigma_\mu).
$$
Hence  $u_\mu$ satisfies the Euler-Lagrange equation \eqref{eq-ut} with $\sigma_\mu$ replaced with $\sigma$, and therefore $\sigma= \sigma_\mu$. 
\end{proof}

\begin{remark}
Obviously, for $\mu=0$ we have $\sigma_0=0$ and  $\Lambda_0=0$.
\end{remark}

\subsection{Asymptotic behavior  of $\Lambda_\mu$ on smooth domains}  

\begin{corollary}\label{thm-smooth}
Let  $\partial\Omega$, and let $\Lambda_\mu$ be given by \eqref{supinf} .  Then
\begin{equation} \label{Lambda-asymp}
\Lambda_\mu = -\frac{\mu^2}{|\partial\Omega|^2}  + \mu\  \frac{N-1}{|\partial\Omega|^2}\, \int_{\partial\Omega} H(y)\,  d\nu(y) +  \mathcal{O}(1) \qquad \mu\to -\infty, 
\end{equation}
\end{corollary}

\begin{proof}
 In view of the regularity of $\Omega$ and  \cite[Thm.~1.1]{vdbg}, 
\begin{equation} \label{heat-content-asymp-bis}
\Big | Q_\Omega(t) - |\Omega| + \frac{2 |\partial\Omega|}{\sqrt{\pi}}\, t^{\frac12}  - \frac{(N-1)\, t}{2}  \int_{\partial\Omega} H(y)\,  d\nu(y)\, \Big |   \leq C\, t^{3/2} \qquad \forall\, t > 0.
\end{equation}
Combining this with \eqref{g-Q} and \eqref{eq-f} gives 
\begin{equation} \label{F-asymp}
F(s)  = - \sqrt{-s} \  |\partial\Omega|  +\frac{N-1}{2}\, \int_{\partial\Omega} H(y)\,    d\nu(y)  +\mathcal{O}(|s|^{-1/2}) \qquad s\to -\infty.
\end{equation}
Now the claim follows from eqautions  \eqref{implicit} and \eqref{max}.
\end{proof}

\begin{remark}\label{rems}  Let us make a couple of comments regarding Corollary \ref{thm-smooth}.
\begin{enumerate}

\item{\bf Constant $\sigma$}. Obviously, $\sigma = \frac{\mu}{|\partial\Omega|} \in \Sigma_\mu$.  Equation \eqref{Lambda-asymp} thus implies, see \eqref{supinf}, 
\begin{align}  \label{sigma-constant-asymp}
\lambda(\sigma) \leq \Lambda_\mu  = -\sigma^2  + \sigma\, \frac{N-1}{|\partial\Omega|}\, \ \int_{\partial\Omega} H(y)\,    d\nu(y) + \mathcal{O}(1) \quad  \text{as} \ \ \sigma\to  -\infty .
\end{align}
This is to be compared with the two-term asymptotic  expansion \eqref{sigma-const-lin}. 
Note that 
$$
\int_{\partial\Omega} H(y)\,    d\nu(y) \leq |\partial\Omega| \, H_{\rm max}.
$$

\item{\bf A ball}. If $\Omega$ is a ball of radius $R$, then $H(s) = H_{\rm max} =R^{-1}$ for any $s\in\partial\Omega$. 
Moreover, from \eqref{max} we deduce that  $\sigma_\mu$ is constant and therefore
$\sigma_\mu= \frac{\mu}{|\partial\Omega|}$.
Hence $\lambda(\sigma) = \Lambda_\mu $, and inequality \eqref{sigma-constant-asymp} turns into an asymptotic equation  
which coincides  with \eqref{sigma-const-lin}. (Notice that the error term in \eqref{sigma-const-lin} is of order  $\mathcal{O}(1)$ if $\Omega$ is a ball, cf.~\cite{HK}). 

\item{\bf $N$-dimensional spherical shell}. Let $\Omega$ be a spherical shell with radii $R$ and $r$. Corollary \ref{thm-smooth} then gives
\begin{equation}
\Lambda_\mu = -\frac{\mu^2}{|\partial\Omega|^2}  + \mu\  (N-1) \frac{ |\Sph_N| }{|\partial\Omega|^2}\, \big ( R^{N-2}+ r^{N-2}\big)  +\mathcal{O}(1).
\end{equation}
\end{enumerate}
\end{remark}


\section{\bf Some related results}
\label{sec-related}

\subsection{Asymptotic behavior of $\Lambda_\mu$ on polygons}  As an example of non-smooth domains with corners, we treat two-dimensional 
 polygons.
 
\begin{corollary}\label{thm-polygon}
Let $\Omega\subset \R^2$ an open, bounded and connected set with polygonal boundary. Denote by $\gamma_j, j=1,\dots,n,$ the internal angles 
of $\partial\Omega$, and let 
\begin{equation}
c(\alpha) = \int_0^\infty  \frac{4 \sinh((\pi-\alpha)x)}{\sinh(\pi x) \cosh(\alpha x)}\, dx \, .
\end{equation}
Then
\begin{equation} \label{Lambda-asymp-polygon}
\Lambda_\mu = -\frac{\mu^2}{|\partial\Omega|^2}  + \frac{2  \mu}{|\partial\Omega|^2} \,  \sum_{j=1}^n c(\gamma_j)+  \mathcal{O}(1) \qquad \mu\to -\infty.
\end{equation}
\end{corollary}

Note that in this case the second equation in \eqref{max} has to be replaced by 
$$
\sigma_\mu(y) = -s(\mu) \, \partial_n U_{s(\mu)}(y)  \qquad  \forall  \ y\in\partial\Omega\setminus \{V\},
$$ 
where $V$ denotes the set of vertices of $\Omega$. Obviously, $\sigma_\mu\in L^\infty(\partial\Omega)$. 

\begin{proof}
By \cite[Thm.~1]{vdbs},
\begin{equation} \label{heat-content-asymp-tris}
\Big |Q_\Omega(t) - |\Omega| + \frac{2 |\partial\Omega|}{\sqrt{\pi}}\, t^{\frac12}  -  t  \sum_{j=1}^n c(\gamma_j)  \Big | \leq  C\, e^{-\frac{\delta}{t}} \qquad \forall\, t >0.
\end{equation}
for some $\delta >0$. Equation \eqref{Lambda-asymp-polygon} thus follows by mimicking the proof of Corollary \ref{thm-smooth}. 
\end{proof}

\begin{remark}  When $\sigma$ is constant, then Corollary \ref{thm-polygon} implies the asymptotic upper bound
\begin{align}  \label{upperb-polygon}
\lambda(\sigma) \leq \Lambda_\mu  = -\sigma^2  + \frac{ 2 \sigma}{|\partial\Omega|} \, \sum_{j=1}^n c(\gamma_j) + \mathcal{O}(1) \qquad \sigma\to  -\infty .
\end{align}
It is again instructive to confront the right hand side in the above equation with the  asymptotic  expansion for $\lambda(\sigma)$ on polygons. 
The latter reads
\begin{align} \label{sigma-constant-lin-polygon} 
\lambda(\sigma) = -\frac{\sigma^2}{\sin^2(\gamma/2)} +  \mathcal{O}(\sigma),  \qquad \gamma :=\min_j \gamma_j 
\end{align}
as $ \sigma\to  -\infty ,$ see \cite{bp,kh,lp}. Notice that  $\sin^2(\gamma/2)<1$. Hence contrary to the case of smooth domains, the asymptotic expansion for $\lambda(\sigma)$ differs from the upper bound 
\eqref{upperb-polygon} already in the leading term. 
\end{remark}

\subsection{Asymptotic behavior of $\Lambda_\mu$ for small $\mu$}  

For the sake of completeness we state a two-term asymptotic expansion of $\Lambda_\mu$ also for $\mu\to 0$.  Notice that the result stated below holds for all Lipschitz regular domains.

Let us denote by $E_1 <E_2 \leq \cdots E_j\leq \cdots$ the eigenvalues of $-\Delta_\Omega^D$, and by $\{\varphi_j\}_{j\geq 1}$ the set of the corresponding eigenfunctions normalized to $1$ in $L^2(\Omega)$.

\begin{corollary}\label{cor-small-mu}
Let $\Omega$ be a bounded connected and set in $\R^N$ with a Lipschitz  boundary.   Then
\begin{equation} \label{Lambda-asymp-small}
\Lambda_\mu = \frac{\mu}{|\Omega|} -\frac{\mu^2}{|\Omega|^2}\,  \sum_{j=1}^\infty\, \frac{\alpha^2_j}{E_j} +\mathcal{O}(\mu^3) \qquad \mu\to 0,
\end{equation}
where
$$
\alpha_j = \int_\Omega \varphi_j(x)\, dx \, .
$$
\end{corollary}

\begin{proof} 
Since $\{\varphi_j\}_{j\geq 1}$ form an orthonormal basis of $L^2(\Omega)$, the Parseval identity implies that 
$$
\sum_{j=1}^\infty \alpha_j^2 = |\Omega|.
$$
Hence the sum on the right hand side of \eqref{Lambda-asymp-small} is convergent. By the definition of $U_s$ we then get
\begin{align*}
s^2 \int_\Omega U_s(x)\, dx & = s^2\, \sum_{j=1}^\infty \,  \frac{\alpha^2_j}{E_j-s}  = s^2\, \sum_{j=1}^\infty \,  \frac{\alpha^2_j}{E_j} +  \mathcal{O}(|s|^3) \qquad s \to 0.
\end{align*}
Equation \eqref{Lambda-asymp-small} thus follows from  \eqref{eq-f}, \eqref{implicit} and \eqref{max}.
\end{proof}

\subsection{The infimum} It is not difficult to verify that the infinitum in \eqref{supinf} is not attained. Indeed, we have 

\begin{proposition} \label{prop-inf1}
Let $N\geq 2$ and let $\mu<0$. Then $ \inf_{\sigma\in \Sigma_\mu} \lambda(\sigma)= -\infty$.
\end{proposition}

\begin{proof}
Let $s_0\in\partial\Omega$ and let $\sigma_n \geq 0$ be given by
$$
\sigma_n(s) = 
\Big\{
\begin{array}{l@{\qquad}l}
a_n    &      {\rm if \ } \  s\in\  B(s_0, 2^{-n})\cap \partial\Omega ,  \\
0     &   {\rm elsewhere \ } ,
\end{array}
\Big. 
$$
where $a_n$ is a positive constant chosen so that $\sigma_n \in\Sigma_\mu$ for all $n\in\N$.  Depending on the dimension we construct a family of test functions $u_n$  as follows:
$$
u_n(x) = \frac{\log \big|\log |x-s_0|\big |}{\log (\log n)} \qquad   \text{on \ } B(s_0, n^{-1})\cap\Omega, \qquad u_n \equiv 1 \quad \text{on \ } \Omega\setminus B(s_0, n^{-1}), \quad 
$$
if $N=2$, and 
$$
u_n(x) = n^{\beta}\, |x-s_0|^{\beta}   \qquad \text{on \ } B(s_0, n^{-1})\cap\Omega, \qquad  u_n \equiv 1 \quad \text{on \ } \Omega\setminus B(s_0, n^{-1}),
$$
if  $N\geq 3$, where $-\frac 12 < \beta <0$. 
Then $u_n \in H^1(\Omega)$ for all $n\in\N$ and a quick calculation shows that 
\begin{align*} 
\int_\Omega |\nabla u_n|^2 \, dx \, \sim\, n^2, \qquad  \int_{\partial\Omega} \sigma_n\, u_n^2\, d\nu\, \sim\, n^{2\beta}\, 2^{-2n\beta },  \qquad n\to \infty, \quad N\geq 3,
\end{align*} 
and
\begin{align*} 
\int_\Omega |\nabla u_n|^2 \, dx \to 0, \qquad  \int_{\partial\Omega} \sigma_n\, u_n^2\, d\nu \to \infty   \qquad n\to \infty, \quad N=  2.
\end{align*} 
Since $u_n\to 1$ uniformly in $\Omega$, this implies 
$$
\lim_{n\to\infty} Q[\sigma_n, u_n] = -\infty. 
$$
Hence the claim.
\end{proof}

\section*{\bf Acknowledgements}

The work was supported by the European Union's Horizon 2020 research and innovation programme under the Marie Sk{\l}odowska-Curie grant agreement No 873071.

\bigskip


\end{document}